\documentclass[11pt]{article}
\usepackage{amsthm, amsfonts,epsf,amsmath,amssymb,tikz,url,listings}
\usepackage{grffile}

\newtheorem{theorem}{Theorem}
\newtheorem{lemma}[theorem]{Lemma}
\newtheorem{prob}{Problem}

\newtheorem{conj}{Conjecture}

\DeclareFixedFont{\ttb}{T1}{txtt}{bx}{n}{12} % for bold
\DeclareFixedFont{\ttm}{T1}{txtt}{m}{n}{12}  % for normal

\usepackage{color}
\definecolor{deepblue}{rgb}{0,0,0.5}
\definecolor{deepred}{rgb}{0.6,0,0}
\definecolor{deepgreen}{rgb}{0,0.5,0}

\newcommand\pythonstyle{\lstset{
tabsize=1,
language=Python,
basicstyle=\ttm,
otherkeywords={sage},             % Add keywords here
keywordstyle=\ttb\color{deepblue},
emph={True},          % Custom highlighting
emphstyle=\ttb\color{deepred},    % Custom highlighting style
stringstyle=\color{deepgreen},
frame=tb,                         % Any extra options here
showstringspaces=false            % 
}}

\lstnewenvironment{python}[1][]
{
\pythonstyle
\lstset{#1}
}
{}

\newcommand\pythoninline[1]{{\pythonstyle\lstinline!#1!}}

\author{
Jernej Azarija \\
Institute of Mathematics, Physics and Mechanics \\
Jadranska 19, 1000 Ljubljana, Slovenia \\
jernej.azarija@imfm.si
}

\date{\today}

\title{Tutte polynomials and a stronger version of the Akiyama-Harary problem} 

\begin{document}
\maketitle

\begin{abstract}
Can a non self-complementary graph have the same chromatic polynomial as its complement? The answer to this question of Akiyama and Harrary is positive and was given by J. Xu and Z. Liu. They conjectured that every such graph has the same degree sequence as its complement. In this paper we show that there are infinitely many graphs for which this conjecture does not hold. We then solve a more general variant of the Akiyama-Harary problem by showing that there exists infinitely many non self-complementary graphs having the same Tutte polynomial as their complements.
\end{abstract}


\noindent {\bf AMS Subj. Class. (2010)}: 05C31, 05C76  

\section{Introduction}

Let $p_G(k)$ be the {\em chromatic polynomial} of a simple graph $G$ that is $p_G(k)$ is the number of proper $k$-colorings of $G.$ In 1980 Akiyama and Harary \cite{AKH} raised the following question `Is there a graph $G$ that is not self-complementary and has a chromatic polynomial that equals to the chromatic polynomial of $\overline{G}$?' Observe that since $p_G(k)$ encodes the number of edges of $G$ a necessary conidition for a graph to have the posed property is that it has precisely ${|V(G)| \choose 2}/{2}$ edges.

The question recived little attention until 1995 when J. Xu and Z. Liu \cite{Xu} showed that such a graph indeed exists. They have shown that for any $n \geq 8$ congurent to $0$ or $1$ modulo $4$ there exists a graph $G$ of order $n$ such that $G$ is not self-complementary and $p_G(k) = p_{\overline{G}}(k).$ In their paper they constructed graphs with a specific degree sequence and then used the degree sequence to compute the chromatic polynomial of the coresponding graph. Given the nature of their construction they posed 

\begin{conj}[J. Xu, Z. Liu]
If a graph $G$ has the property that $p_{G}(k) = p_{\overline{G}}(k)$ then $G$ has the same degree sequence as $\overline{G}.$
\end{conj}

As it turns out, their conjecture is false. In this paper we present an infinite family of graphs not adhering to this condition. 

Finally we turn our attention to a more general variant of the problem introduced by Akiyama and Harary. For a subset $F \subseteq E(G)$ we denote by $c(F)$ the number of connected components of the graph with edge set $F$ and vertex set $V(G).$ With this in mind the {\em Tutte polynomial} of a graph $G$ is defined as \begin{equation} \label{TutteEq} T_G(x,y) = \sum_{F \subseteq E(G)} (x-1)^{c(F) - c(E)} \cdot (y-1)^{c(F)+|F|-|V(G)|} .\end{equation}

The Tutte polynomial $T_{G}(x,y)$ contains much more information about the structure of $G$ than $p_{G}(k)$ does. Indeed, it is well known that $$p_G(k) = (-1)^{|V(G)|-k(E)}k^{c(E)}T_{G}(1-k,0).$$ Among the many other interesting evaluations of the Tutte polynomial are $T_G(1,1)$ - the number of spanning trees of $G$ and $T_G(2,0),T_G(0,2)$ the number of cyclic and acycic orientations of $G$ respectively. For a survey of known results about the Tutte polynomial see \cite{SomeTutteSurvey}. 

A natural generalization of the Harary-Akiyama question following from these properties of the Tutte polynomial is, wheter there exists non self-complementary graphs having the same Tutte polynomial as their complement. In this paper we shall prove

\begin{theorem}
There exists infinitely many graphs that are not self-complementary and have the same Tutte polynomial as their complement.
\end{theorem}

\section{Chromatic polynomials and graph complements}

In this section we present a family of graphs having equal chromatic polynomials as their complements but different degree sequence. We start with the graph depicted on Figure \ref{G1} together with its complement. Its {\em graph6} string \cite{McKay} is \verb+HCpVdZY+. First, we establish that $G$ has the desired properties.

\begin{lemma} \label{L0}
There exists a graph $G$ of order 9 such that $G$ and $\overline{G}$ have different degree sequences but $p_{G}(k) = p_{\overline{G}}(k).$
\end{lemma}

\begin{proof}
We observe that the graph $G$ from Figure \ref{G1} has degree sequence $(5, 5, 5, 4, 4, 4, 4, 3, 2)$ while its complement has degree sequence $(6, 5, 4, 4, 4, 4, 3, 3, 3).$ Using the well known {\em deletion-contraction} recurrence for computing the chromatic polynomial of a graph we can verify that $$p_{G}(k) = p_{\overline{G}}(k) = (k - 2) \cdot (k - 1) \cdot k \cdot (k - 3)^{2} \cdot (k^{4} - 9k^{3} + 35k^{2} - 69k + 57).$$ Alternatively we can verify the stated claim using the Sage program presented in the Appendix.
\end{proof}

\begin{figure}
\centering
\includegraphics[scale=0.55]{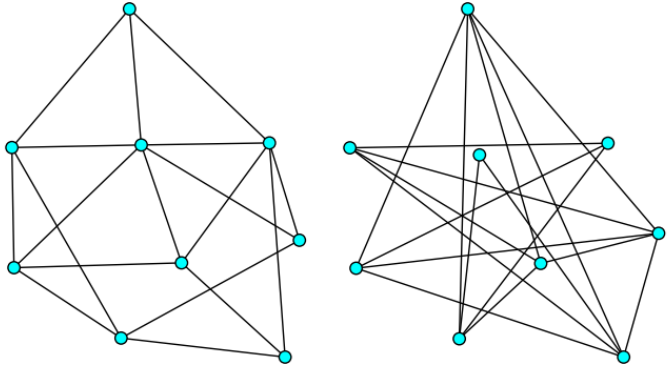}
\caption{A graph and its complement.} \label{G1}
\end{figure}

Before showing the main claim of this section, we introduce a useful construction. Given a graph $G$ we form the graph $\widehat{G}$ by taking a vertex disjoint $4$-path $P$ and joining every vertex of $G$ to both endpoints of $P.$ Conveniently, we have $\overline{\widehat{G}} = \widehat{\overline{G}}.$ Using this property it is not difficult to establish the following claim.

\begin{theorem} \label{T1}
There exists infinitely many graphs $G$ not having the same degree sequence as $\overline{G}$ but having the same chromatic polynomial as their complements. 
\end{theorem}

\begin{proof}   
We compute the chromatic polynomial of $\widehat{G}.$  Suppose we wish to properly color $\widehat{G}$ with $k$ colors. Let $x,y$ be the endpoints of the $4$-path $P$  introduced in $\widehat{G}$ and let $x',y'$ be the respective neighbors of $x$ and $y$ in $P.$ There are essentialy two different ways to color $\widehat{G}.$ If we color $x,y$ with equal colors then there are $(k-1)$ choices to color $x'$ and $(k-2)$ colors to color $y'$ and hence $k(k-1)(k-2) p_G(k-1)$ ways to properly $k$-color $\widehat{G}.$ If $x,y$ are colored with different colors then we again have two cases. If $y'$ is colored with the same color as $x$ then we have $k(k-1)^2 p_G(k-2)$ total ways to color $\widehat{G}.$ If however $y'$ is not colored with the same color as $x$ we end up having $k(k-1)(k-2)^2 p_G(k-2)$ ways to propery color our graph using $k$ colors. Summing up the obtained quantities we infer \begin{eqnarray} p_{\widehat{G}}(k) & = & k(k-1)(k-2) p_G(k-1)+ k(k-1)^2 p_G(k-2) + k(k-1)(k-2)^2 p_G(k-2) \nonumber \\
       & = & k(k-1)( (k-2) p_G(k-1) + (k(k-3)+3)p_G(k-2)) \nonumber . \end{eqnarray}
In particular we see from the above expression that $\widehat{G}$ is in fact a function of $p_G(k).$  The main claim now follows quickly with an inductive argument. By Lemma 2 we have a graph $G$ of order 9 having a different degree sequence than $\overline{G}$ but the same chromatic polynomial. But then the degree sequences of $\widehat{G}$ and $\overline{\widehat{G}}$ differ while for their chromatic polynomials the above identity implies \begin{eqnarray} p_{\widehat{G}}(k) & = & k(k-1)( (k-2) p_G(k-1) + (k(k-3)+3)p_G(k-2)) \nonumber \\ & = & k(k-1)( (k-2) p_{\overline{G}}(k-1) + (k(k-3)+3)p_{\overline{G}}(k-2)) \nonumber \\ &  = & p_{\widehat{\overline{G}}}(k) = p_{\overline{\widehat{G}}}(k). \nonumber \end{eqnarray} Hence by using this construction iteratively we obtain an infinite family of graphs with the stated property.
\end{proof}

Making a computer search it can be seen that there are graphs on 12 vertices that have the property stated in Theorem \ref{T1}. Hence it is easy to extend the proof of Theorem \ref{T1} to show that for any $n \geq 9$ congurent to $0$ or $1$ $\pmod{4}$ there exist a graph $G$ not having the same degree sequence as $\overline{G}$ but sharing the same chromatic polynomial.

\section{The Tutte polynomial}

A very useful property of the chromatic polynomial that we exploited in the proof of Theorem \ref{T1} is the fact that the chromatic polynomial of a graph operation is often a function of the chromatic polynomials of its operands. Unfortunately the same is not generally true for the Tutte polynomial. Indeed, consider two trees of order $4$, the star graph $K_{1,3}$ and the path graph $P_4.$ Both have the same Tutte polynomial namely $x^3.$ Consider now their {\em cone graph}, that is the graph obtained by adding a new vertex and joining it to all other vertices. The cone of $K_{1,3}$ has $20$ spanning trees while while the cone of $P_4$ has $21$ spanning trees. Hence the Tutte polynomials of the cones of $K_{1,3}$ and $P_4$ are different. 

In order to apply the construction introduced in the previous section, we need an additional structure of our graphs that will assure that if two graphs $G$ and $H$ have equal Tutte polynomials then so do $\widehat{G}$ and $\widehat{H}.$

As it turns out, the following concept is quite useful for this purpose. Let $H$ be a spanning subgraph of $G$ having connected components of order $h_1 \geq h_2 \geq \cdots \geq h_k.$ We say that $(|E(H)|,h_1,h_2, \ldots, h_k)$ is a {\em subgraph description} of $H.$ Let now $s(G)$ be the lexicographically sorted tuple of subgraph descriptions for every subgraph of $G.$ We call $s(G)$ the {\em subgraph sequence} of $G.$
Observe that equation \ref{TutteEq} implies that if two graphs have the same subgraph sequence then they also have the same Tutte polynomial. The converse is of course not true as witnessed by the above example with $P_4$ and $K_{1,3}.$ Our next lemma asserts that the property of having the same subgraph sequence is preserved by the construction introduced in the previous section.

\begin{lemma} \label{L1}
If $G$ and $H$ are graphs such that $s(G) = s(H)$ then $s(\widehat{G}) = s(\widehat{H}).$
\end{lemma}

\begin{proof}
Let $G'$ be a spanning subgraph of $\widehat{G}.$ Observe that $G'$ is obtained by taking a spanning subgraph of $G$ with subgraph description $d = (|E(G')|, g_1,\ldots,g_k)$ adding the remaining four vertices of $\widehat{G}$ comming from the introduced $4$-path $P$ and finally adding some of the edges with at least one endpoint in $P.$ That is we add some of the edges of $P$ and then some of the edges from the endpoints of $P$ to some vertices of the connected components of $G.$

By assumption $G$ has the same subgraph sequence as $H$ hence there is a bijective mapping between their subgraph sequences. Let $H'$ be the subgraph of $H$ with subgraph sequence $d$ that is prescribed by such bijection. Since $H'$ and $G'$ have the same subgraph description there is bijective way to map every extension of $G'$ to a subgraph of $\widehat{G}$ to an extension of $H'$ to a subgraph of $\widehat{H}.$ Indeed, we may assume the vertices of $G$ and $H$ to be ordered and then for every edge that is added from one of the endpoints $x$ of $P$ to the the $i$'th vertex of the $j$'th component of $G$ we add the edge between $x$ and the $i$'th vertex of the $j$'th component of $H.$ This is always well defined since $H$ and $G$ have the same subgraph description.
\end{proof}

In order to apply Lemma \ref{L1} we need to find a non self-complementary graph $G$ such that $s(G) = s(\overline{G}).$ As already noted this immediately implies $T_{G}(x,y) = T_{\overline{G}}(x,y).$  One of the smallest graphs with such property has order $8$ and is presented on Figure \ref{G2}.  Its graph6 string is \verb+GCRdvK+ . 

\begin{figure}
\centering
\includegraphics[scale=0.55]{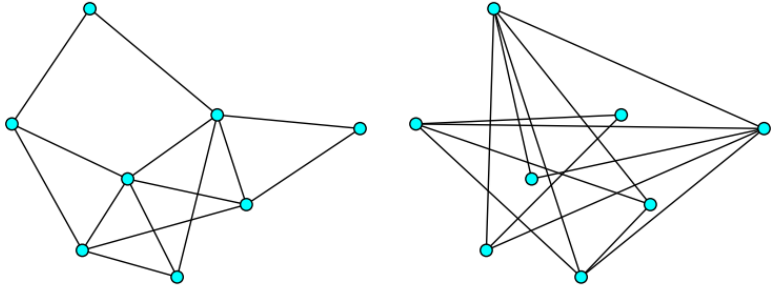}
\caption{A graph with equal Tutte polynomial as its complement.} \label{G2}
\end{figure}

\begin{lemma} \label{L2}
There exist a non self-complementary graph of order $8$ such that $s(G) = s(\overline{G}).$
\end{lemma}

\begin{proof}
Consider the graph $G$ from Figure \ref{G2}. Observe that $G$ and $\overline{G}$ both have two vertices of degree $2.$ In $G$ these two vertices share a common neighbor while the vertices of degree $2$ in $\overline{G}$ have no common neighbors. Hence $G$ and $\overline{G}$ are not isomorphic. To verify the second part of the claim, that is $s(G) = s(\overline{G}),$ is a tedious process hence we invite the reader to inspect Appendix \ref{AppA} presenting a Sage program verifying the claim.
\end{proof}

We are now ready to prove the main claim of this section.

\begin{theorem}\label{TT}
There exist infinitely many graphs $G$ such that $G \not \cong \overline{G}$ but $T_{G}(x,y) = T_{\overline{G}}(x,y).$
\end{theorem}

\begin{proof}
By Lemma \ref{L2} there is a non self-complementary graph on $9$ vertices such that $s(G) = s(\overline{G})$ which implies $T_G(x,y) = T_{\overline{G}}(x,y).$ But then, by Lemma \ref{L1} the graph $\widehat{G}$  again has the same subgraph description as its complement and is not self-complementary. Hence applying this operation iteratively on $G$ we end up with an infinite family of graphs possesing the stated property.
\end{proof}

Again as with the chromatic polynomial we can find a graph of order $9$ having the properties of Lemma \ref{L2}. Hence it is possible to show in the same way as we did in the proof of Theorem \ref{TT} that for any $n \geq 8$ congurent to $0,1$ modulo $4$ there exist a non self-complementary graph of order $n$ having the same Tutte polynomial as its complement.

\section{Final remarks}

We were not able to find an example of a graph $G$ with different degree sequence from $\overline{G}$ but same Tutte polynomial. A computer search indicates that such a graph would have to have at least $16$ vertices. Hence we leave the following problem.

\begin{prob}
Find a graph $G$ with different degree sequence than $\overline{G}$ but same Tutte polynomial or show that such a graph does not exists.
\end{prob}

Interestingly the equivalent problem for chromatic polynomials motivated this paper.

\section{Acknowledgements}
We are thankful to thank Sandi Klav\v{z}ar for constructive remarks to Gordon Royle and Georgi Guninski for help with some computational aspects of the problem and to Nejc Trdin for kindly sharing his computational resources. 

\appendix
\section{Sage programs used in the proofs} \label{AppA}

In this appendix we show how to prove the claims of Lemmas \ref{L1}, \ref{L2} using the open source mathematical software Sage \cite{sage}. All examples can be directly copy-pasted into Sage's shell. 

In order to prove  Lemma \ref{L0}, we need to verify that the presented graph has a different degree sequence than its complement but equal chromatic polynomial.
\newline
\begin{python}
sage: G = Graph('HCpVdZY')                                 
sage: Gc = G.complement()                                  
sage: G.degree_sequence() == Gc.degree_sequence()          
False
sage: G.chromatic_polynomial() == Gc.chromatic_polynomial()
True
\end{python}

To check Lemma \ref{L1} we need to first define a function computing the subgraph description of a graph.

\begin{python}
def s(Gr):
    ds = []

    for A in subsets(Gr.edges()):
        G = Graph()
        G.add_vertices(Gr.vertices())
        G.add_edges(A)
        cs = [H.order() for H in G.connected_components_subgraphs()]
        ds.append([len(A)] + sorted(cs)) 

    return sorted(ds)
\end{python}

It is now a matter of a few lines to verify Lemma \ref{L1}.

\begin{python}
sage: G = Graph('GCRdvK')
sage: Gc = G.complement()
sage: G.is_isomorphic(Gc)
False
sage: s(G) == s(Gc)
True
\end{python}

\end{document}